\documentclass[11pt]{amsart}
\usepackage[utf8]{inputenc}
\usepackage[english]{babel}
 
\usepackage[square,sort,comma,numbers]{natbib}
\usepackage{soul}
\usepackage{geometry}

\usepackage[dvipsnames]{xcolor}
\usepackage{float}

\usepackage{graphicx}
\usepackage{tikz-cd}
\usetikzlibrary{arrows}
\usetikzlibrary{matrix,arrows,decorations.pathmorphing}

\usepackage{amsmath}
\usepackage{amssymb}
\usepackage{amsthm}

\usepackage{stmaryrd}
\usepackage[all]{xy}
\usepackage[mathcal]{eucal}
\usepackage{verbatim}\usepackage{fullpage} \usepackage{wrapfig}
\usepackage{lipsum}
\usepackage[all]{xy}
\theoremstyle{plain}
\newtheorem{thm}{Theorem}[section]

\newtheorem{lemma}[thm]{Lemma}

\newtheorem{prop}[thm]{Proposition}

\theoremstyle{definition}
\newtheorem{defn}[thm]{Definition}

\theoremstyle{remark}

\newcommand{\nc}{\newcommand}
\nc{\dmo}{\DeclareMathOperator}

\DeclareMathOperator{\Conf}{Conf}

\DeclareMathOperator{\Mod}{Mod}
\DeclareMathOperator{\SMod}{SMod}

\DeclareMathOperator{\Homeo}{Homeo}
\DeclareMathOperator{\SHomeo}{SHomeo}

\nc{\para}[1]{\medskip\noindent\textbf{#1.}}
\title{Non-realizability of some big mapping class groups}

\address{Department of Mathematics   \newline University of Maryland, College Park }
\email{chenlei@umd.edu }
\address{Department of Mathematics   \newline University of Oklohoma, Norman, OK, 73019 }
\email{he@ou.edu }
\author{Lei Chen and Yan Mary He}
\begin{document}
 \bibliographystyle{alpha}
\maketitle
\begin{abstract}
In this note, we prove that the compactly supported mapping class group of a surface containing a genus $3$ subsurface has no realization as a subgroup of the homeomorphism group. We also prove that for certain surfaces with order $6$ symmetries, their mapping class groups have no realization as a subgroup of the homeomorphism group. Examples of such surfaces include the plane minus a Cantor set and the sphere minus a Cantor set.
\end{abstract}

\section{introduction}
In this paper, we study the generalized Nielsen realization problem for mapping class groups of infinite type surfaces. Let $X$ be a surface. We say that $X$ is {\it finite type} if its fundamental group $\pi_1(X)$ is finitely generated and $X$ is {\it infinite type} if $\pi_1(X)$ is not finitely generated. The mapping class group  $\Mod(X):=\pi_0(\Homeo^+(X))$ of $X$ is the group of path-connected components of the orientation preserving homeomorphism group $\Homeo^+(X)$ of $X$. In particular, $\Mod(X)$ is called a {\it big} mapping class group if $X$ is an infinite type surface.

For any surface $X$, there is a natural forgetful homomorphism \[
	p_X:\Homeo^+(X)\to \Mod(X)
	\] sending an orientation preserving homeomorphism of $X$ to its mapping class. We say that $p_X$ has a {\it section} if there is a homormorphism $\rho : \Mod(X) \to \Homeo^+(X)$ such that the composition $p_X \circ \rho$ equals the identity map.
The generalized Nielsen realization problem asks if the forgetful map $p_X$
has a section or not. The importance of this problem in geometric topology is its relation to the existence of flat structures on the universal $X$-bundle. We will give more details in Section 2. In general, we expect that $p_X$ not to have sections if the fundamental group $\pi_1(X)$ is not abelian. 

For  finite type surfaces, Nielsen posed the realization question for finite subgroups of $\Mod(X)$ in 1943 and Kerckhoff \cite{Ker} showed that a lift always exists for finite subgroups of $\Mod(S_g)$,  where $S_g$ is the closed surface of genus $g \ge 2$. The first result on the Nielsen realization problem for the whole mapping class group is a theorem of Morita \cite{Mor} that when $g \ge 18$, there is no section for the projection $p_{S_g}$ to the group of orientation preserving $C^2$-diffeomorphisms. Then Markovic \cite{Mar} showed that $p_{S_g}$ does not have a section for $g\ge 6$ (Chen--Salter \cite{CS} later gave an elementary proof of the same result with the lowest possible genus bound). Franks--Handel \cite{FH}, Bestvina--Church--Suoto \cite{BCS} and Salter--Tshishiku \cite{ST} have also obtained the non-realization theorems for $C^1$-diffeomorphisms. Recently, Chen \cite{Chen} proved that braid group $B_n$ has no realization in the group of homeomorphisms for $n\ge6$ and Chen--Markovic \cite{CM} proved that the Torelli group has no realization in the group of area-preserving homeomorphisms. The current big open problem in this area is whether every finite index subgroup of the mapping class group has realization as homeomorphisms or not. We refer the readers to the survey paper of Mann--Tshishiku \cite{MT} for more history and previous results.

For {\it infinite type} surfaces, the finite subgroup realization has been established in a recent paper of Afton--Calegari--Chen--Lyman \cite{Lvzhou}. More precisely, they showed that any finite subgroup of the big mapping class group $\Mod(X)$ can be realized as a subgroup of homeomorphisms of $X$ where $X$ is an orientable infinite type surface with or without boundary. In this paper, we study the Nielsen realization problem for the {\it compactly supported} mapping class group $\Mod_c(X)$ where $X$ is any surface containing a genus 3 subsurface. Removing the compact support condition, we also study the Nielsen realization problem for the entire big mapping class group $\Mod(X)$ for certain surfaces $X$ which admit an order 6 symmetry.

\subsection{Statement of results}
Let $X$ be a surface. Recall that a homeomorphism of $X$ is {\it compactly supported} if it is the identity map outside of a compact subset of $X$. Let $\Homeo^+_c(X)$ be the group of orientation preserving compactly supported homeomorphisms of $X$. The compactly supported mapping class group $\Mod_c(X)$ of $X$ is the connected component group of $\Homeo^+_c(X)$. We refer the interested reader to the survey paper of Aramayona--Vlamis \cite[Chapter 2.4]{bigSurvey} for more details.

Our first theorem concerns realizability of $\Mod_c(X)$ where $X$ is a surface containing a genus $3$ subsurface.
\begin{thm}\label{main2}
	If $X$ is a surface which contains a genus $3$ subsurface, then the projection map
	\[
	p_S^c:\Homeo^+_c(X)\to \Mod_c(X)
	\]
	has no sections.
\end{thm}

To prove the theorem, we use Markovic's minimal decomposition theory \cite{Mar}. Suppose on the contrary that there exists a section $\rho_c : \Mod_c(X) \to \Homeo_c^+(X)$. Then using Markovic's minimal decomposition theory, we find a subsurface of $X$ having genus $3$ and one boundary component on which we can realize the mapping class group $\Mod(S_{3,1})$ as a subgroup of $\Homeo^+(S_{3,1})$. Here $\Homeo^+(S_{3,1})$ is the group of orientation preserving homeomorphisms of a genus $3$ surface fixing a point and $\Mod(S_{3,1})$ is the connected component group of $\Homeo^+(S_{3,1})$. However, this is a contradiction to the following theorem.

\begin{thm}\label{genus3}
	The natural projection map $p_{S_{3,1}}:\Homeo^+(S_{3,1})\to \Mod(S_{3,1})$ has no sections.
\end{thm}
To prove this theorem, we use the Birman-Hilden theory and a theorem of the first author's that $\Mod(S_{0,7,1})$ has no realizations in $\Homeo^+(S_{0,7,1})$. Here $\Mod(S_{0,7,1})$ is the mapping class group of a surface of genus 0 fixing 2 sets of points, where one set contains 7 points and the other set contains one point.

If we remove the compact support condition in Theorem \ref{main2} and consider the Nielsen realization problem for the entire big mapping class group, it seems impossible that Markovic's theory can be directly applied to this non-compactly supported case. This is due to the fact that the shadowing property of local Anosov map fails to work if the realization of a mapping class does not have compact support. 

However, for some special infinite type surfaces, we do obtain non-realizability results for the entire big mapping class group.

\begin{thm}\label{cantor}
Let $X = \mathbb{R}^2 \setminus \mathcal{C}$ or $X = \mathbb{S}^2 \setminus \mathcal{C}$ where $\mathcal{C}$ is a set which can be arranged as in Figure \ref{fig1} so that Lemma \ref{lem_C} holds. Then the map $p_X:\Homeo^+(X)\to \Mod(X)$ has no sections.
\end{thm}
Examples of such set $\mathcal{C}$ include the Cantor set and 6 copies of a cardinal set $\omega^{\alpha}+1$. A key consequence of this choice of $\mathcal{C}$ is that the surface $X = \mathbb{R}^2 \setminus \mathcal{C}$ (or  $X = \mathbb{S}^2 \setminus \mathcal{C}$) admits an order 6 symmetry.

To prove the theorem, we use this order 6 mapping class to construct a finitely generated subgroup $\Gamma$ of $\Mod(X)$ that has no realization as homeomorphisms of $X$. More specifically, suppose on the contrary that there exists a section $\rho : \Mod(X) \to \Homeo^+(X)$. Then we will show that the realization of every element in $\Gamma$ has the same set of fixed points which does not contain any point in $\mathcal{C}$. However, there exists an order 5 element in $\Gamma$ whose realization fixes a point in $\mathcal{C}$, which is a contradiction.

We point out that if $X$ is the plane minus a Cantor set, then any countable subgroup of $\Homeo^+(S^1)$ can be embedded into $\Homeo^+(X)$ by Calegari--Chen in the proof of \cite[Theorem 5.1]{CalegariChen}. In particular, any countable subgroup of $\Mod(X)$ can be embedded into $\Homeo^+(X)$. Therefore, the finitely generated subgroup $\Gamma$ that we construct in Section 3 has an embedding into  $\Homeo^+(X)$ but such an embedding is not a realization. In contrast, $\Mod(S_g)$ for a genus $g>2$ surface has no embedding into $\Homeo^+(S_g)$ by Franks--Handel \cite[Corollary 1.3]{FH2}.

%%\lei{Our next result concerns realizability of compactly supported mapping class groups \cite[Chapter 2.4]{bigSurvey}. Let $\Homeo^+_c(X)$ be the group of orientation preserving compactly supported homeomorphisms of $X$ and let $\Mod_c(X)$ be the connected component group of $\Homeo_c(X)$.
%%\begin{thm}
%	If $X$ contains a genus $3$ subsurface, then the projection
%\[
%p_S^c:\Homeo^+_c(X)\to \Mod_c(X)
%\]
%has no section.
%\end{thm}
%The proof of this theorem intensely uses techniques in \cite{Mar}. It seems impossible that such techniques can be applied to the non-compactly supported case due to the fact that the shadowing property of local Anosov map fails to work if the realization of a mapping class does not have compact support. 
%}

\subsection{Organizaiton of the paper} The paper is organized as follows. In Section 2, we give geometric interpretation of the generalized Nielsen realization problem. In Section \ref{sec_Mar}, we give a brief overview of Markovic's minimal decomposition theory. We prove Theorem \ref{main2} and Theorem \ref{genus3} in Section \ref{sec_compact}. Finally, we prove Theorem \ref{cantor} in Section \ref{sec_Cantor}.

\subsection*{Acknowledgement} We would like to thank Danny Calegari for suggesting \cite[Theorem 5.1]{CalegariChen} and Jing Tao for useful comments.

\section{Geometric interpretation}
In this section, $X$ can be any manifold. The main motivation to study the generalized Nielsen realization problem is its connection to $X$-bundles. 

%A {\it principal $\Homeo^+(X)$-bundle} is a fiber bundle $\pi : E \to B$ whose fiber is a $\Homeo^+(X)$-space and whose structure group is $\Homeo^+(X)$. The space $B\Homeo^+(X)$ is defined up to homotopy by the fact that there is a principal $\Homeo^+(X)$-bundle $P \to B\Homeo^+(X)$ with $P$ contractible. The space $B\Homeo^+(X)$ is called the {\it classifying space} of $\Homeo^+(X)$. 

An {\it $X$-bundle} is a fiber bundle $\pi: E \to B$ whose fiber is the manifold $X$ and whose structure group is $\Homeo^+(X)$. In particular, there is an open cover $\{U_i\}$ for the base $B$ such that the bundle is trivial on each $U_i$, i.e. $\pi^{-1}(U_i) = U_i \times X$ and the trivial bundles are glued by transition functions $U_i \cap U_j \to \Homeo^+(X)$.
The {\it universal $X$-bundle} $E\Homeo^+(X)$ is an $X$-bundle whose base is the {\it classifying space} $B\Homeo^+(X)$. The universality of the bundle is given in the following sense. Given any map from a CW-complex $B$ to $B\Homeo^+(X)$, we obtain an $X$-bundle $E \to B$ by pulling back the universal $X$-bundle as shown in the following diagram. 
\begin{center}
\begin{tikzcd}
	E \arrow[r] \arrow[d]
	& E\Homeo^+(X)  \arrow[d ] \\
	B \arrow[r, black]
	& |[black, rotate=0]| B\Homeo^+(X)
\end{tikzcd}
\end{center}

Let $\Homeo^+(X)^{\delta}$ be the discrete group of $\Homeo^+(X)$. Then $\Homeo^+(X)^{\delta}$ also has its classifying space $B\Homeo^+(X)^{\delta}$. We have the following natural forgetful map
\[
\pi_X:  B\Homeo^+(X)^\delta \to B\Homeo^+(X).
\]
Let $\widetilde{B}$ be the universal cover of $B$. We say that an $X$-bundle over $B$ has a {\it topological flat structure} if this bundle is $\widetilde{B}\times X/\pi_1(B)$ where the action is given by deck transformations on $\widetilde{B}$ and a homomorphism $\pi_1(B)\to \Homeo^+(X)$ on $X$. For this reason, whether the universal $X$-bundle has a flat structure corresponds to whether the map $\pi_X$ has a section. However, a section $s$ of $\pi_X$ gives the following homomorphism by taking fundamental groups 
\[
\pi_1(B\Homeo^+(X))=\pi_0(\Homeo^+(X))=\Mod(X) \xrightarrow{s_*} \Homeo^+(X) = \pi_1(B\Homeo^+(X)^{\delta}).
\]
Therefore, if the map $s$ is a section of $\pi_X$ then the map $s_*$ is a section of $\pi_{X*} = p_X$. This means that a negative answer of the generalized Nielsen realization problem for $X$ implies that the universal $X$ bundle has no flat structure. In general, it is hard to compute both the homology and the homotopy groups of $B\Homeo^+(X)$ and $B\Homeo^+(X)^\delta$.

\section{Markovic's minimal decomposition theory} \label{sec_Mar}
In this section, we give a brief overview of Markovic's minimal decomposition theory. We start with the definition of an {\it upper semi-continuous decomposition} of a surface.
\begin{defn}[{\cite[Definition 2.1]{Mar}}]
	Let $M$ be a surface. A collection ${\bf S}$ of closed connected subsets of $M$ is an {\it upper semi-continuous decomposition of $M$} if the following holds:
	\begin{enumerate}
		\item If $S_1, S_2 \in {\bf S}$, then $S_1 \cap S_2 = \emptyset$.
		\item If $S \in {\bf S}$, then $S$ does not separate $M$; that is $M \setminus S$ is a connected set.
		\item We have $M = \cup_{S \in {\bf S}} S$.
		\item If $S_n \in {\bf S}$ for $n \in \mathbb N$, is a sequence that has the Hausdorff limit $S_0$, then there exists $S \in {\bf S}$ such that $S_0 \subset S$.
	\end{enumerate}
\end{defn}

\begin{defn}[{\cite[Definition 2.2]{Mar}}]
	Let $M$ be a surface and let $S \subset M$ be a closed, connected subset of $M$ which does not separate $M$. We say that $S$ is {\it acyclic} if there is a simply connected open set $U \subset M$ such that $S \subset U$ and $U \setminus S$ is homeomorphic to an annulus.
\end{defn}

Let ${\bf S}$ be an upper semi-continuous decomposition of $M$. For every point $p \in M$, there exists a unique $S_p \in {\bf S}$ that contains $p$. We denote by $M_{\bf S}$ the set of all points $p \in M$ such that the corresponding $S_p$ is acyclic.

Now we add some dynamics to an upper semi-continuous decomposition of $M$ by considering a subgroup $G$ of $\Homeo^+(M)$. In particular, we need the following definition for an upper semi-continuous decomposition of $M$ to be compatible with $G$.

\begin{defn}[{\cite[Definition 3.1]{Mar}}]
	Let ${\bf S}$ be an upper semi-continuous decomposition of $M$. Let $G$ be a subgroup of $\Homeo^+(M)$. We say that ${\bf S}$ is {\it admissible for the group $G$} if the following holds:
	\begin{enumerate}
		\item Each $\hat{f} \in G$ preserves setwise every component of ${\bf S}$.
		\item Let $S \in {\bf S}$. Then every point in every frontier component of the surface $M \setminus S$ is a limit of points from $M \setminus S$ that belong to acyclic components of ${\bf S}$ (note that not every point of $S$ need to be in a frontier component of the subsurface $M \setminus S$).
	\end{enumerate}
	If $G$ is the cyclic group generated by a homeomorphism $\hat{f} : M \to M$, we say that ${\bf S}$ is an admissible decomposition for $\hat{f}$.
\end{defn}

An admissible decomposition ${\bf S}$ for a group $G$ will be called the {\it minimal decomposition for $G$} if ${\bf S}$ is contained in every admissible decomposition for $G$.
\begin{thm}[{\cite[Theorem 3.1]{Mar}}]
	Every subgroup $G$ of $\Homeo^+(M)$ has a unique minimal decomposition.
\end{thm}

\section{Realization of compactly supported big mapping class groups} \label{sec_compact}
In this section, we prove Theorem \ref{main2} which states that the projection map
\[
p_S^c:\Homeo^+_c(S)\to \Mod_c(S)
\]
has no sections if the surface $S$ contains a subsurface of genus $3$. 

We first prove a preliminary result, namely Theorem \ref{genus3}, which states that the natural projection map $$p_{S_{3,1}}:\Homeo^+(S_{3,1})\to \Mod(S_{3,1})$$ has no sections. Recall that $\Homeo^+(S_{3,1})$ is the orientation preserving homeomorphism group of genus $3$ fixing a point and $\Mod(S_{3,1})$ is the connected component group of $\Homeo^+(S_{3,1})$. 

\begin{proof}[Proof of Theorem \ref{genus3}]
	Let $x\in S_3$ be the marked point on a genus 3 surface $S_3$. Let $\tau$ be the hyper-elliptic involution of $S_3$ where one of its fixed points is $x$. Let $\SHomeo^+(S_{3,1})$ be the symmetric homeomorphism group; i.e., the group of homeomorphisms of $S_3$ fixing $x$ that commute with $\tau$. Let $\SMod(S_{3,1})$ be the symmetric mapping class group under $\tau$ defined as the image of $\SHomeo(S_{3,1})$ in $\Mod(S_{3,1})$. Then by the Birman--Hilden theory \cite{BH} (see also \cite[Theorem 9.2]{FM}), we know that the symmetric mapping group satisfies the following short exact sequence
	\[
	1\to \langle \tau\rangle\to \SMod(S_{3,1})\to  \Mod(S_{0,7,1})\to 1
	\]
	where $\Mod(S_{0,7,1})$ is the mapping class group of surface of genus $0$ with two sets of marked points where one set contains one point and the other set contains 7 points. If the projection map $p_{S_{3,1}}$ had a realization, say $\rho$, then by Ahlfors’ trick, $\rho(\tau)$ is conjugate to the standard hyper-elliptic involution. It follows that $\tau$ has $8$ fixed points, one of which is $x$. A section of $p_{S_{3,1}}|_{\SMod(S_{3,1})}$ gives a section of the map $p_{S_{0,7,1}} : \Homeo^+(S_{0,7,1})\to \Mod(S_{0,7,1})$. However, this is a contradiction to \cite[Theorem 1.1]{Chen}.
\end{proof}

We now prove Theorem \ref{main2}. We remark that we will use intensely results, definitions and arguments in \cite{Mar}, which will all work if all the elements used are compactly supported. However, if we try to argue for the non-compactly supported case, even the shadowing property of the Anosov element does not work, and seems impossible to repair. 

\begin{proof}[Proof of Theorem \ref{main2}]
	We prove the theorem by contradiction. Suppose that there exists a section
	\[
	\rho_c: \Mod_c(S)\to \Homeo^+_c(S)
	\]
	of $p_S^c$.  Let $T\subset S$ be the subsurface of genus $3$ with one boundary component. Then there is an embedding 
	\[
	i:\Mod(T)\to \Mod_c(S).
	\]
	We will show that $p_S^c$ already does not have sections over $i(\Mod(T))$. Let $c$ be the boundary component of $T$ and let $T_c$ denote the Dehn twist about the curve $c$.
	
	Let $M_c$ be the union of acyclic components of the minimal decomposition for $\rho(T_c)$.
	Then by \cite[Proposition 4.1]{Mar} which works for compactly supported homeomorphisms of infinite type surfaces, we know that $M_c$ contains a subsurface $T'$ of $S$, which contains a subsurface of genus at least $1$. Since $T_c$ is a Dehn twist which commutes with all the Dehn twists about curves disjoint from $c$, we know that the end of $T'$ cannot be any curve other than $c$. On the other hand, $T'$ cannot be the whole surface, because otherwise $\rho_c(T_c)$ is homotopic to identity by \cite[Theorem 2.1]{Mar} (see a more complete argument in \cite[Lemma 5.1]{Mar}). Then we know that $T'$ is homotopic to $T$.
	
	Since $\Mod(T)$ commutes with $T_c$, we know that $\rho_c(\Mod(T))$ preserves $T'$ and is a permutation of all the acyclic components of $\rho_c(T_c)$ inside $T'$. Now we have an action $\rho'$ of $\Mod(T)$ on $T'/\sim$ where two points are equivalent if they belong to the same acyclic set in the minimal decomposition for $\rho_c(T_c)$. Under this action, the element $\rho'(T_c)$ is identity. Thus this gives a realization of the homomorphism $p_{S_{3,1}}:\Homeo^+(S_{3,1})\to \Mod(S_{3,1})$, which does not exist by Theorem \ref{genus3}.
\end{proof}

\section{The proof of Theorem \ref{cantor}} \label{sec_Cantor}
In this section, we prove Theorem \ref{cantor}. We first prove the theorem for $X = \mathbb{R}^2 \setminus \mathcal{C}$ or $X = \mathbb{S}^2 \setminus \mathcal{C}$ where $\mathcal{C}$ is a Cantor set. Then it will be immediate from the proof that the essential property of the set $\mathcal{C}$ which makes the argument work is Lemma \ref{lem_C}. Therefore as long as $\mathcal{C}$ can be arranged as shown in Figure 1 and satisfies Lemma \ref{lem_C}, Theorem \ref{cantor} is true. 

\subsection{The case of a Cantor set}
Let $X$ be the plane minus a Cantor set or the sphere minus a Cantor set. We first construct a specific subgroup of $\Mod(X)$ and the relations in it. Then we show that this special subgroup cannot act as homeomorphisms.

Let $C$ be the standard middle-third Cantor set embedded in the plane $\mathbb R^2$. We construct a set of points $C'$ in the plane as follows. Let $D_1,...,D_6$ be the discs as shown in Figure 1. We denote by $P_1,...,P_6$ the centers of the discs $D_1,...,D_6$ respectively. Inside each disk $D_i$, there is a wedge union $C_i$ of $30$ copies of the standard Cantor set with the center $P_i$ as the wedge point. Let $C'$ be the union of $C_i$. 

\begin{figure}[H]   
\includegraphics[scale=0.22]{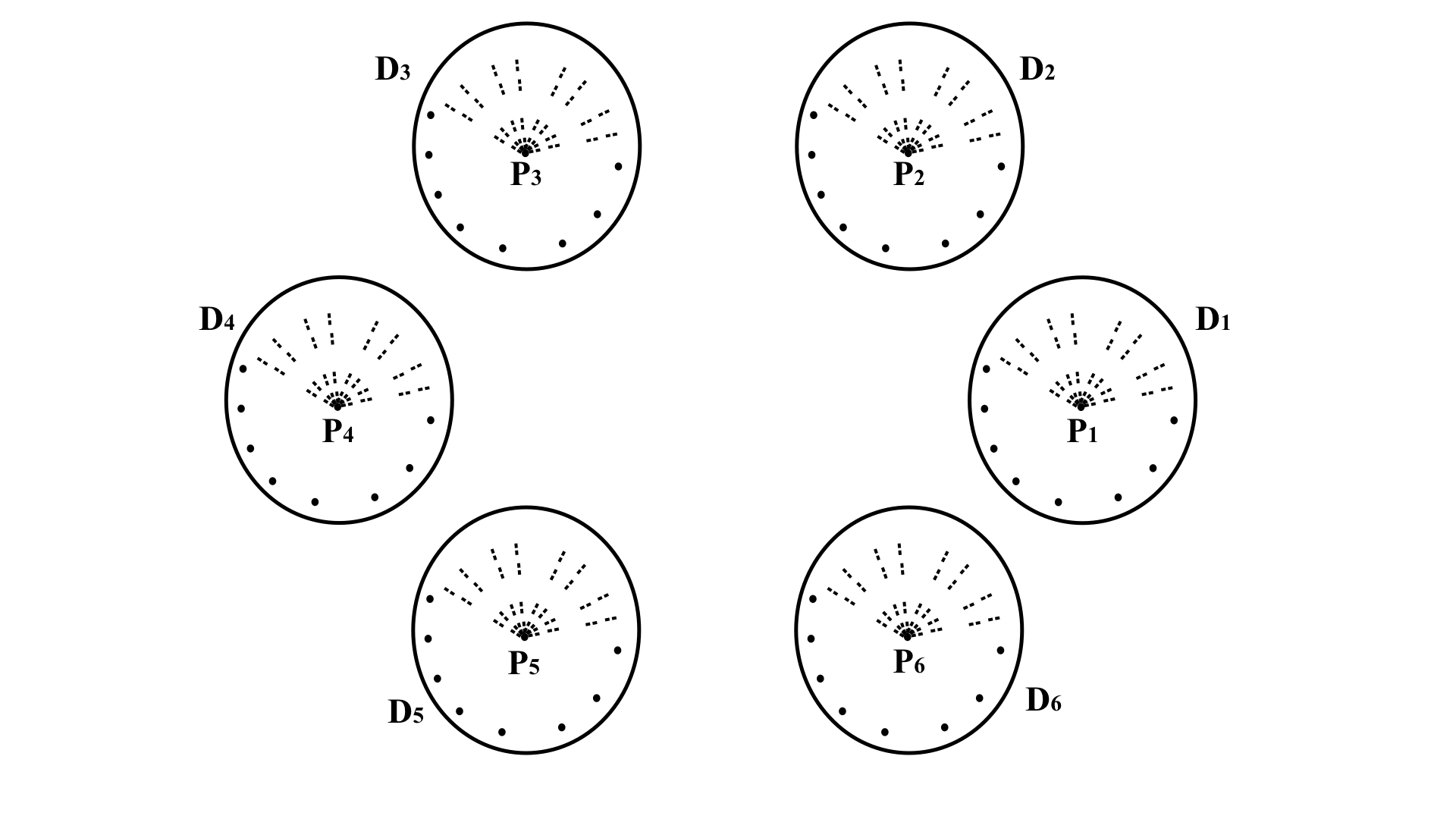}
\label{fig1}
\caption{Configuration of the Cantor set $C'$}
\end{figure}

We have the following lemma.
\begin{lemma}\label{lem_C}
The two surfaces $\mathbb R^2 - C'$ and $\mathbb R^2 - C$ are homeomorphic.
\end{lemma}
\begin{proof}
We first show that $C'$ is homeomorphic to the Cantor set $C$. Brower's theorem states that a topological space is homeomorphic to the Cantor set $C$ if and only if it is non-empty, perfect, compact, totally disconnected, and metrizable. By construction, $C'$ is clearly non-empty, perfect, compact, totally disconnected. Moreover, $C'$ is metrizable as it is embedded in $\mathbb R^2$. Therefore, there exists a homeomorphism $\varphi : C' \to C$. By Schoenflie's theorem, $\varphi $ extends to a homeomorphism $\tilde{\varphi} : \mathbb R^2 \to \mathbb R^2$ with $\tilde{\varphi}(C') = C$. It follows that $\mathbb R^2 - C'$ is homeomorphic to $\mathbb R^2 - C$.
\end{proof}

Now we will study the generalized Nielsen realization problem for $\mathbb{R}^2-C'$. Denote by $R(\theta)$ the simultaneous counter-clockwise rotation of each $D_i$ by an angle of $\theta$. To make $R(\theta)$ a homeomorphism of $\mathbb{R}^2-C'$, we need to apply a twist back on an annulus around the disk $D_i$. Around each $D_i$, Figure 2 represents this homeomorphism. Based on our construction of $C'$, we will consider the angle $\theta$ as an integer multiple of $2\pi/30$.

\begin{figure}[H]  
	\includegraphics[scale=0.2]{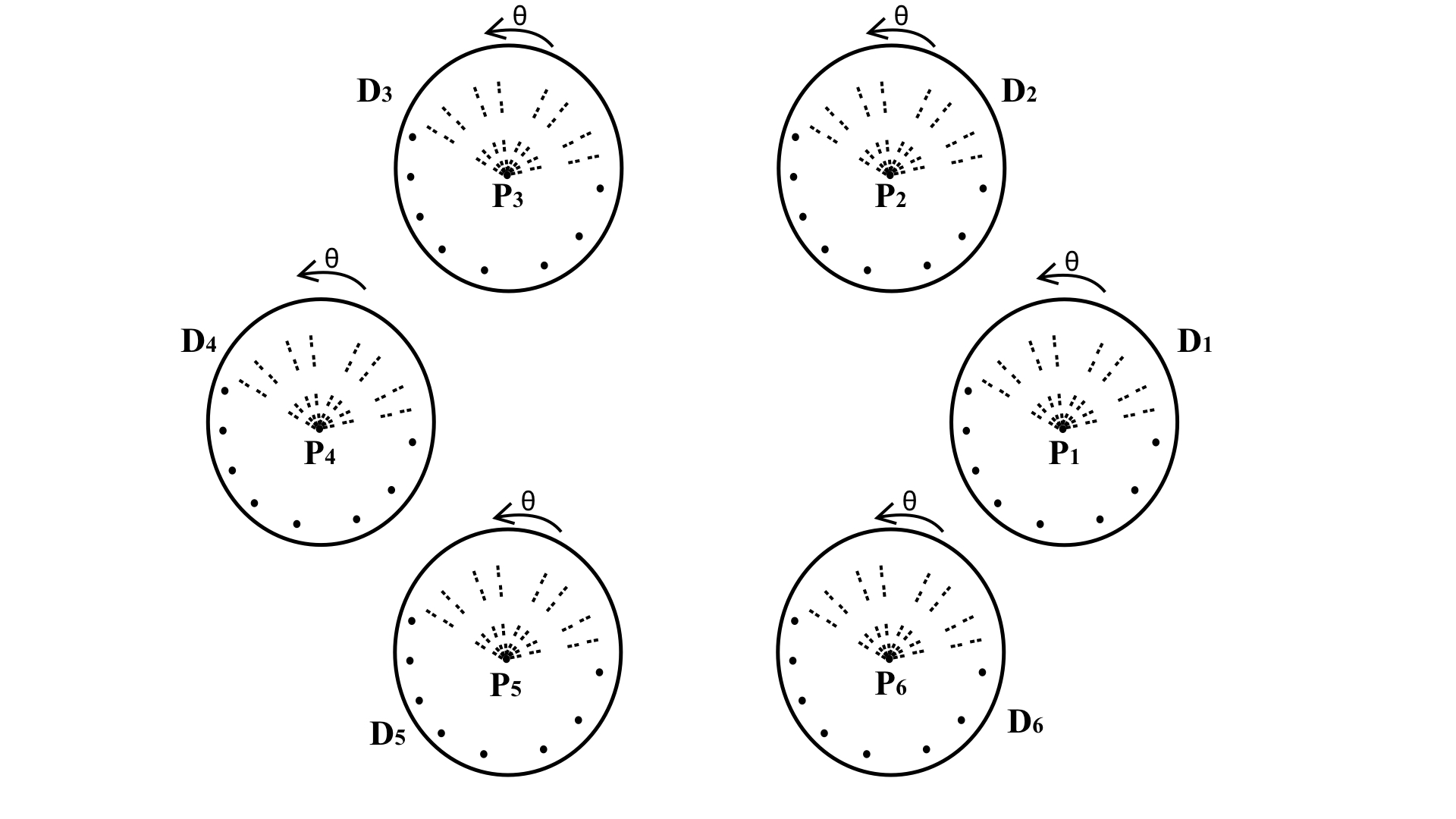}
	\label{fig2}
	\caption{Rotation $R(\theta)$}
\end{figure}

\begin{defn}[Rigid half twist]
Let $\sigma_i$ be the disc pushing between $D_i$ and $D_{i+1}$ clockwise as shown in Figure 3. We call $\sigma_i$ the rigid half twist between $D_i$ and $D_{i+1}$. \end{defn}

Loosely speaking, the map $\sigma_i$ moves the discs around in a specified way while keeping the interior of each disc fixed relative to the center of the disc.

\begin{figure}[H]  
	\includegraphics[scale=0.2]{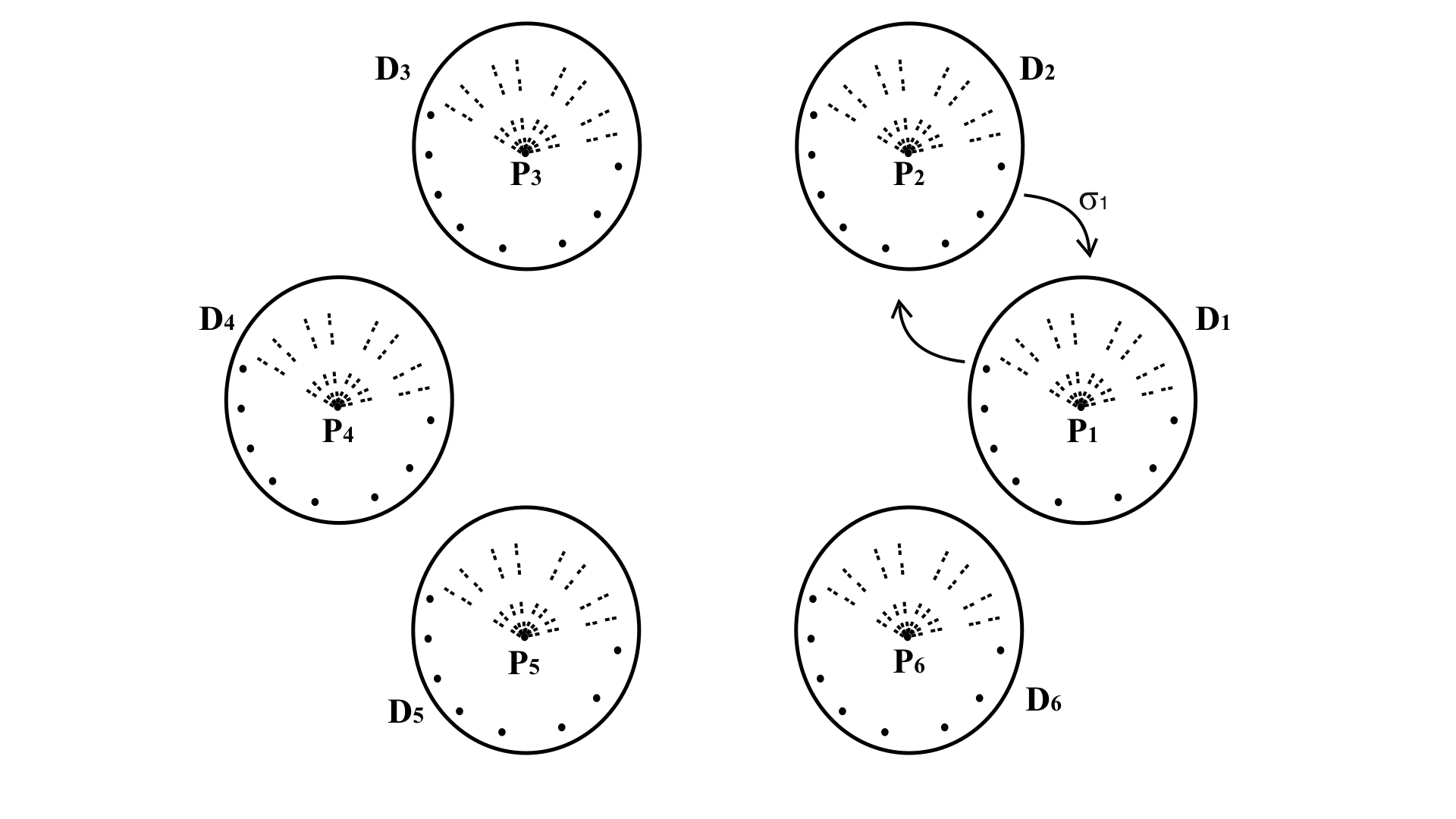}
	\label{fig3}
	\caption{Disk pushing $\sigma_1$}
\end{figure}

Let $\alpha_1$ be the clockwise rotation by $2\pi/6$ of the plane. By symmetry, we see that $\alpha_1$ preserves $C'$ and is thus an element in $\Mod(\mathbb{R}^2-C')$. Moreover, $\alpha_1$ is an order $6$ element.
\begin{defn}[Rigid rotation]
We define {\it rigid rotation}, denoted by $\alpha_1^r$, to be the disc pushing following the path as shown in Figure 4.
\end{defn}

\begin{figure}[H]  
	\includegraphics[scale=0.2]{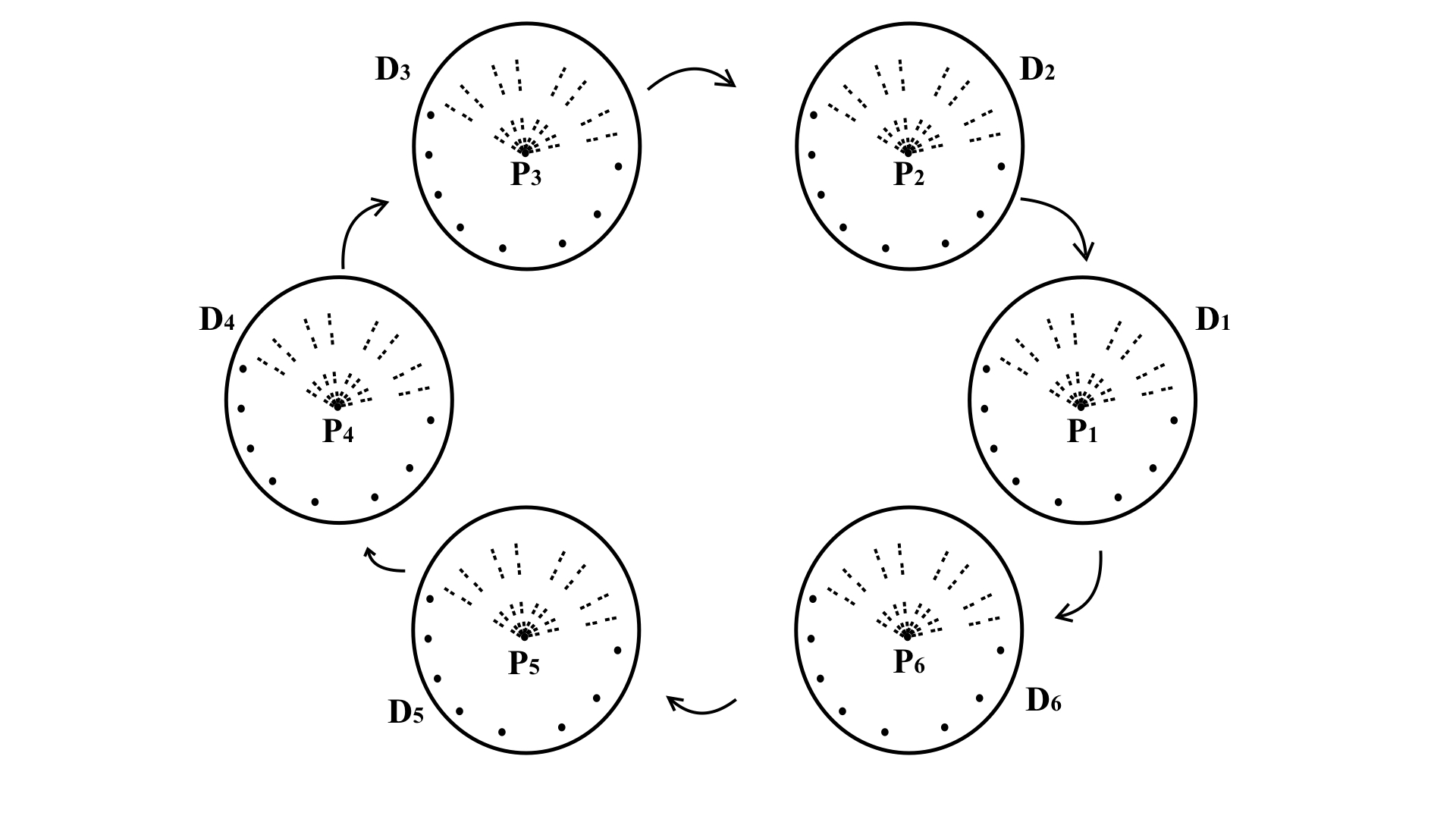}
	\label{fig4}
	\caption{Rigid rotation $\alpha_1^r$}
\end{figure}

Define $\alpha_2:=\alpha_1\sigma_5R(2\pi/30)$. We have the following relations of all the elements we have defined.
\begin{lemma}\label{comp}
We have the following relations:
\begin{enumerate}
\item
$R(2\pi/30)$ commutes with $\sigma_i$,
\item
$\alpha_1(\sigma_i)\alpha_1^{-1} = \sigma_{i+1}$,
\item
$\alpha_1^r=\sigma_1\sigma_2\sigma_3\sigma_4\sigma_5=\sigma_6\sigma_1\sigma_2\sigma_3\sigma_4$,
\item
$\alpha_1 = \alpha_1^r R(2\pi/6)$,
\item
$\alpha_2$ is an order $5$ element.
\end{enumerate}
\end{lemma}
\begin{proof}
\begin{enumerate}
\item
The element $\sigma_i$ preserves the support of $R(2\pi/30)$. Outside the support of $R(2\pi/30)$, they commute because $R(2\pi/30)$ is the identity. Inside the support of $R(2\pi/30)$, they commute because $\sigma_i$ is rigid.
\item
This one is given by definition.
\item
This is a consequence of the same result in the braid group. The point-pushing version is given by the following fiber bundle
\[
\Homeo(\mathbb{R}^2,6 pts)\to \Homeo(\mathbb{R}^2)\to \Conf_6(\mathbb{R}^2)
\]
where $\Homeo(\mathbb{R}^2,6pts)$ is the set of homeomorphisms of $\mathbb{R}^2$ fixing $6$ points as a set and $\Conf_6(\mathbb{R}^2)$ is the configuration space of $6$ distinct unordered points in $\mathbb R^2$.

Taking the long exact sequence of homotopy groups, we obtain a map 
\[
\pi_1(\Conf_6(\mathbb{R}^2))\to \pi_0(\Homeo(\mathbb{R}^2,6pts)).
\]
Figure 5 shows the path for the point-pushing of $\sigma_1\sigma_2...\sigma_5$, from which we can see that the two paths for the point-pushing of $\alpha_1^r$ and $\sigma_1\sigma_2...\sigma_5$ are homotopic to each other.
\begin{center}
\begin{figure}[h!]  
	\includegraphics[scale=0.17]{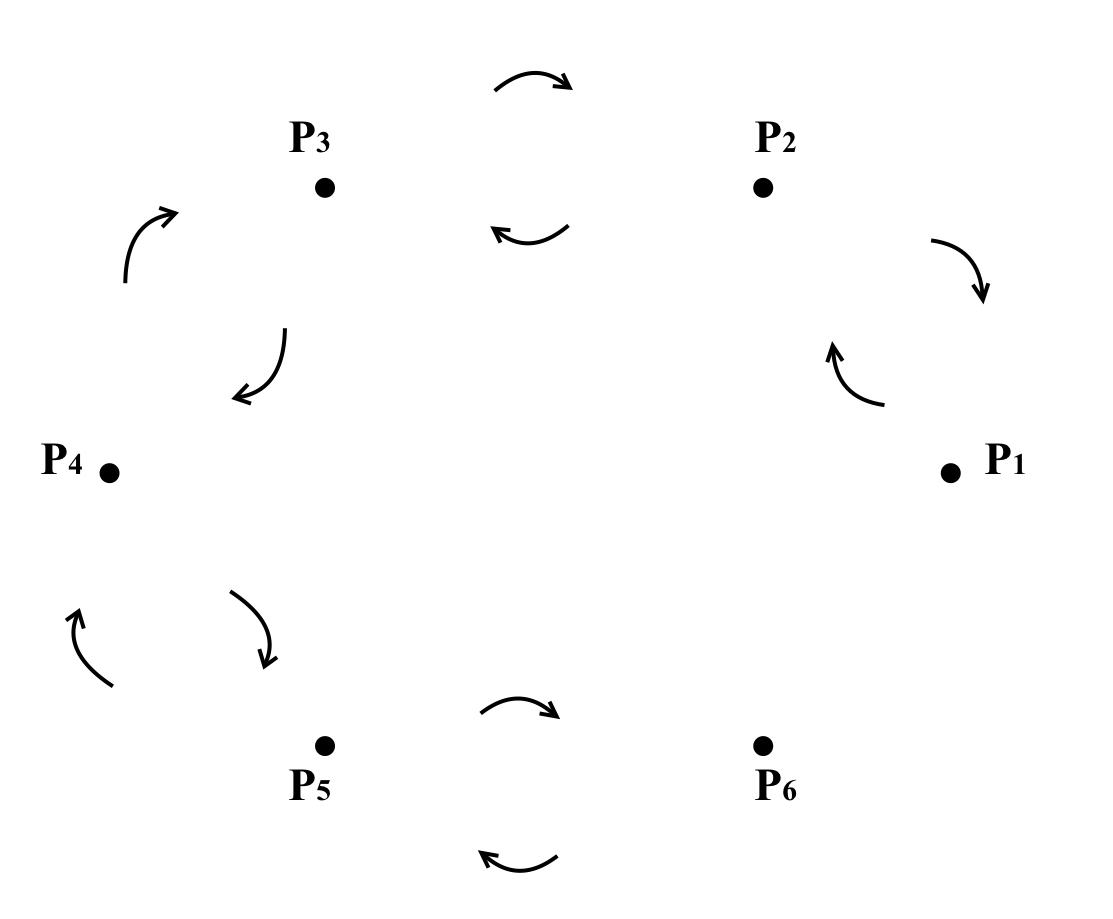}
	\label{fig5}
	\caption{The path of $\sigma_1\sigma_2...\sigma_5$}
\end{figure}
\end{center}

In our case, we are dealing with disc-pushing but since all of our mapping classes are rigid, we obtain the same result.
\item
This is given by the definition of $\alpha_1^r$.
\item
Let us compute $\alpha_2^5$.
\begin{equation*}
  \begin{aligned}
    \alpha_2^5  & = \alpha_1\sigma_5R(2\pi/30)\alpha_1\sigma_5R(2\pi/30)\alpha_1\sigma_5R(2\pi/30)\alpha_1\sigma_5R(2\pi/30)\alpha_1\sigma_5R(2\pi/30) \\
                & = \alpha_1\sigma_5\alpha_1\sigma_5\alpha_1\sigma_5\alpha_1\sigma_5\alpha_1\sigma_5R(2\pi/6)\\
                & = \sigma_6\alpha_1^2\sigma_5\alpha_1\sigma_5\alpha_1\sigma_5\alpha_1\sigma_5R(2\pi/6)\\
                & = \sigma_6\sigma_1\alpha_1^3\sigma_5\alpha_1\sigma_5\alpha_1\sigma_5R(2\pi/6)\\
                &=...=\sigma_6\sigma_1\sigma_2\sigma_3\sigma_4\alpha_1^5R(2\pi/6)\\
                & = \alpha_1^r\alpha_1^{-1}R(2\pi/6)\\
                & = R(-2\pi/6)R(2\pi/6) = 1
  \end{aligned}
\end{equation*}
Therefore $\rho(\alpha_2)$ is an order $5$ element
\end{enumerate}

\end{proof}
Let $\Gamma$ be the subgroup of $\Mod(\mathbb{R}^2-C')$ generated by $\sigma_i$ and $R(2\pi/30)$. We will prove the following proposition which implies Theorem \ref{cantor}.
\begin{prop}
The group $\Gamma$ has no section under $p_X$.
\end{prop}
\begin{proof}
Suppose that $\Gamma$  has a section $\rho$ of $p_X$. Then $\rho(\alpha_1)$ is an order $6$ homeomorphism of $\mathbb{R}^2$, which has a unique fixed point $O$. By the construction of $\alpha_1$, we know that $O$ is not a puncture.

Now we will show that $\{\rho(\sigma_i)\}$ and $\rho(R(2\pi/30))$ both fix $O$. Since $R(2\pi/30)$ commutes with $\alpha_1$, we know that $\rho(R(2\pi/30))$ fixes $O$. The reason that $\{\rho(\sigma_i)\}$ fixes $O$ follows from the same argument as in \cite{Chen}. We include it here for completeness.

We prove this by explicitly writing $\sigma_3$ as a product of elements in $C(\alpha_1^2)$ and $C(\alpha_1^3)$. Here $C(\alpha_1^j)$ is the set of elements in $\Mod(X)$ which commute with $\alpha_1^j$ for $j = 2, 3$.
Let $G$ be the group generated by $C(\alpha_1^2)$ and $C(\alpha_1^3)$. Since $\rho(C(\alpha_1^3))$ and $\rho(C(\alpha_1^2))$ both fix $O$, every element of $\rho(G)$ fixes $O$. We first observe that $\sigma_1\sigma_4,\sigma_2\sigma_5, \sigma_3\sigma_6 \in C(\alpha_1^3)$ and $\sigma_1\sigma_3\sigma_5,\sigma_2\sigma_4\sigma_6\in C(\alpha_1^2)$. Now, we start with 
\[\alpha_1^r=\sigma_1\sigma_2\sigma_3\sigma_4\sigma_5 \in G.\]
Since $\sigma_5\sigma_2\in G$, we have that 
\[
\sigma_1\sigma_2\sigma_3\sigma_4\sigma_5(\sigma_5\sigma_2)^{-1} \in G.\]
By commutativity of $\sigma_2$ and $\sigma_4$, we obtain
\[
\sigma_1\sigma_2\sigma_3\sigma_2^{-1}\sigma_4 \in G.\]
Applying the same calculation for $\sigma_1\sigma_4\in G$, we obtain
\[
\sigma_1\sigma_2\sigma_3\sigma_2^{-1}\sigma_1^{-1}\in G.
\]
Since $\sigma_1\sigma_3\sigma_5\in G$, we obtain
\[
(\sigma_1\sigma_3\sigma_5)^{-1}\sigma_1\sigma_2\sigma_3\sigma_2^{-1}\sigma_1^{-1}(\sigma_1\sigma_3\sigma_5)\in G.
\]
But we know that $\sigma_5$ commutes with every other element in the above equation, so we obtain
\[
\sigma_3^{-1}\sigma_2\sigma_3\sigma_2^{-1}\sigma_3\in G.\]
Since $\sigma_3\sigma_6\in G$, we obtain
\[
(\sigma_3\sigma_6)\sigma_3^{-1}\sigma_2\sigma_3\sigma_2^{-1}\sigma_3(\sigma_3\sigma_6)^{-1}\in G.\]
But we know that $\sigma_6$ commutes with every other element in the above equation, so we obtain
\[\sigma_2\sigma_3\sigma_2^{-1}\in G.\]
Since $\sigma_2\sigma_5\in G$, we obtain
\[(\sigma_2\sigma_5)^{-1}\sigma_2\sigma_3\sigma_2^{-1}(\sigma_2\sigma_5)\in G.\]
But we know that $\sigma_5$ commutes with every other element in the above equation, so we obtain
\[\sigma_3\in G.\] 
By symmetry, we know that $\sigma_i\in G$. Since elements of $\rho(G)$ fixes $O$, we know that $\rho(\alpha_2)$ also fixes $O$, where $\alpha_2 := \alpha_1\sigma_5R(2\pi/30)$. However, by definition of $\alpha_2$, the element $\rho(\alpha_2)$ also fixes a puncture $P_6$, the center of $D_6$. This contradicts the fact that $\rho(\alpha_2)$ is a finite order homeomorphism of $\mathbb{R}^2$, which has a unique fixed point.

For the sphere minus a Cantor set case, we construct the same group as before. Now $\rho(\alpha_1)$ has two fixed points $N$ and $S$ and we use the same argument to show that $\rho(G)$ will fix $\{N,S\}$ as a set. However, as an order $5$ element, $\rho(\alpha_2)$ has to fix $\{N,S\}$ individually. This contradicts the fact that $\rho(\alpha_2)$ also fixes a puncture (which is the center of a disc) as in the case of the plane minus a Cantor set.
\end{proof}

\subsection{Other cases}
We begin with an example where $\mathcal{C}$ is the union of 6 copies of the set $A$ of cardinal $\omega^{\alpha}+1$. For each copy, we wedge 30 copies of $A$ at the largest accumulation point $P_i$ and make them contained in the disc $D_i, i=1,...,6$, similar to the Cantor set in Figure 1. We denote by $\mathcal{C}'$ this particular configuration of these 180 copies of $A$. Then we see that Lemma \ref{lem_C} holds; namely, $\mathbb R \setminus \mathcal{C}$ is homeomorphic to $\mathbb R \setminus \mathcal{C}'$. This is because there exists a homeomorphism $\varphi : \mathcal{C}' \to \mathcal{C}$ which extends by Schoenflie's theorem, to a homeomorphism $\tilde{\varphi} : \mathbb R^2 \to \mathbb R^2$ with $\tilde{\varphi}(\mathcal{C} ') = \mathcal{C}$. From here, the rest of the argument goes through. 

Hence, if we can arrange the set $\mathcal{C}$ into the configuration as shown in Figure 1 so that Lemma \ref{lem_C} holds, then Theorem \ref{cantor} holds.

    	\bibliography{braidlifting}{}
	\bibliographystyle{alpha}

\end{document}